\documentclass[12pt,reqno]{amsart}

\usepackage{amssymb,fullpage,mathrsfs,verbatim,graphicx,paralist}
\usepackage[breaklinks]{hyperref}

\newcommand\remove[1]{}

\renewcommand{\le}{\leqslant}
\renewcommand{\ge}{\geqslant}
\renewcommand{\leq}{\leqslant}
\renewcommand{\geq}{\geqslant}

\newcommand{\R}{\mathbb{R}}

\newcommand{\E}{\mathbb{E}}
\newcommand{\U}{\mathscr{U}}
\newcommand{\F}{\mathcal{F}}

\newcommand{\N}{\mathbb{N}}

\newcommand{\Lip}{\mathrm{Lip}}

\newcommand{\MET}{{\tt MET}}

\DeclareMathOperator{\dist}{dist}

\newtheorem{theorem}{Theorem}

\newtheorem{lemma}[theorem]{Lemma}

\newtheorem{proposition}[theorem]{Proposition}

\newtheorem{corollary}[theorem]{Corollary}

\theoremstyle{definition}

\newtheorem{remark}[theorem]{Remark}

\newtheorem{question}[theorem]{Question}
\newcommand{\eqdef}{\stackrel{\mathrm{def}}{=}}

\graphicspath{{figs/}}

\begin{document}

\title{A note on dichotomies for metric transforms}

\author{Manor Mendel}
\thanks{M.M. was supported by ISF grant 221/07,
BSF grant 2006009, and
a gift from Cisco Research Center}
\address {Mathematics and Computer Science Department\\
The Open University of Israel}
\email{manorme@openu.ac.il}.
\author{Assaf Naor}
\thanks{A.N. was supported
by NSF grants CCF-0635078 and CCF-0832795, BSF grant 2006009, and
the Packard Foundation.}
\address{Courant Institute\\ New York University}
\email{naor@cims.nyu.edu}
\date{}

\begin{abstract}
We show that for every nondecreasing concave function $\omega:[0,\infty)\to [0,\infty)$ with $\omega(0)=0$, either every finite metric space embeds with distortion arbitrarily close to $1$ into a metric space of the form $(X,\omega\circ d)$ for some metric $d$ on $X$, or there exists $\alpha=\alpha(\omega)>0$ and $n_0=n_0(\omega)\in \N$ such that for all $n\ge n_0$, any embedding of $\{0,\ldots,n\}\subseteq \R$ into a metric space of the form $(X,\omega\circ d)$ incurs distortion at least $n^\alpha$.
\end{abstract}

\maketitle

\section{Introduction}

The distortion of a bi-Lipschitz embedding $f:X \to Y$ of metric spaces $(X,d_X)$ and $(Y,d_Y)$ is defined as
$$
\dist(f)\eqdef\left(\sup_{\stackrel{x,y\in X}{x\neq y}}\frac{d_Y(f(x),f(y))}{d_X(x,y)}\right)\cdot\left(\sup_{\stackrel{x,y\in X}{x\neq y}}\frac{d_X(x,y)}{d_Y(f(x),f(y))}\right)=\|f\|_{\text{Lip}} \cdot \|f^{-1}\|_{\text{Lip}}.
$$
If $(X,d_X)$ admits a bi-Lipschitz embedding into $(Y,d_Y)$, then the least distortion of such an embedding is denoted
$$
c_Y(X)\eqdef \inf\left\{\dist(f):\ f:X\to Y\right\}.
$$
If $\F$ is a family of metric spaces then the distortion of $(X,d_X)$ in $\F$ is defined as
$$
c_\F(X)\eqdef \inf\left\{c_Y(X):\ Y\in \F\right\}.
$$

Let $\omega:[0,\infty)\to [0,\infty)$ be a nondecreasing concave function with $\omega(0)=0$. Then $(X,\omega\circ d_X)$ is a metric space for any metric space $(X,d_X)$, known as the $\omega$-metric transform of $(X,d_X)$. When $\omega(t)=t^\theta$ for some $\theta\in (0,1]$, the metric space $(X,d_X^\theta)$ is known as the $\theta$-snowflake of $(X,d_X)$. In what follows we denote the class of all finite metric spaces by $\MET$, and the class of all $\omega$-metric transforms by $$\omega(\MET)\eqdef\{(X,\omega\circ d_X):\ (X,d_X)\in \MET\}.
$$

Let $P_n=\{0,1,\ldots,n\}\subseteq \R$ be the $(n+1)$-point path, equipped with the metric inherited from the real line. The main purpose of this note is to prove the following dichotomy theorem for metric transforms:
\begin{theorem}\label{thm:omega}
For every nondecreasing concave function $\omega:[0,\infty)\to [0,\infty)$ with $\omega(0)=0$, one of the following two dichotomic possibilities must hold true:
\begin{itemize}
\item either for every $X\in \MET$ we have $c_{\omega(\MET)}(X)=1$,
\item or there exists $\alpha=\alpha(\omega)>0$ and $n_0=n_0(\omega)\in \N$ such that for all $n\ge n_0$ we have $c_{\omega(\MET)}(P_n)\ge n^\alpha$.
\end{itemize}
\end{theorem}
Theorem~\ref{thm:omega} is sharp due to the following fact:
\begin{proposition} \label{prop:transform-dichotomy-tight}
	For every $\alpha\in (0,1]$ there exists a nondecreasing and concave function $\omega:[0,\infty)\to [0,\infty)$ with $\omega(0)=0$
such that $c_{\omega(\MET)}(X)\le(|X|-1)^\alpha$ for every finite metric space $X$, and
$c_{\omega(\MET)}(P_n)= n^\alpha$ for every $n\in \mathbb N$.
\end{proposition}

\subsection*{The metric cotype dichotomy problem}

\newcommand{\G}{\mathcal{G}}
 \newcommand{\D}{\mathscr D}
 \renewcommand{\E}{\mathcal E}

Our motivation for proving Theorem~\ref{thm:omega} is one of the main questions left open in our investigation of metric cotype~\cite{MN-cotype-full}. To explain it, we recall the following theorem\footnote{This phenomenon was first conjectured to hold true by  Arora, Lov\'asz, Newman, Rabani, Rabinovich and Vempala in~\cite{ALNRRV}.} from~\cite[Thm.~1.6]{MN-cotype-full}. Given two classes of metric spaces $\E,\F$ and an integer $n\in \N$, denote
$$
\D_n(\E\hookrightarrow \F)\eqdef \sup_{\substack{X\in \E\\|X|\le n}} c_\F(X).
$$
When $\E=\MET$ we write
$$
\D_n(\F)\eqdef \D_n(\MET\hookrightarrow \F)= \sup_{\substack{X\in \MET\\|X|\le n}} c_\F(X).
$$

\begin{theorem}[Metric cotype dichotomy~\cite{MN-cotype-full}]\label{thm:MN dich}
For any class of metric spaces $\F$, one of the following two dichotomic
possibilities must hold true:
\begin{itemize}
\item either $\D_n(\F)=1$ for all $n\in \N$,
\item or there exists $\alpha=\alpha(\F)>0$ and $n_0=n_0(\F)\in \N$ such that for all $n\ge n_0$ we have $\D_n(\F)\ge (\log n)^\alpha$.
\end{itemize}
\end{theorem}
We call Theorem~\ref{thm:MN dich} a ``metric cotype dichotomy" since the parameter $\alpha=\alpha(\F)$ is related in~\cite{MN-cotype-full} to a numerical invariant of the class $\F$ called the {\em metric cotype of $\F$}; we refer to~\cite{MN-cotype-full} for more details since we will not use the notion of metric cotype here.

A consequence of Theorem~\ref{thm:MN dich} is that if we were told that $\D_n(\F)=(\log n)^{o(1)}$ then we would immediately deduce that actually $\D_n(\F)=1$ for all $n\in \N$. The theory of metric dichotomies studies such dichotomic behavior (if it exists) of the rate of growth of $\{\D_n(\E\hookrightarrow \F)\}_{n=1}^\infty$. For example, we have the following classical result of Bourgain, Milman and Wolfson~\cite{BMW}, corresponding to the case when $\E$ consists of all Hamming hypercubes, and $\F$ consists of a single metric space.
\begin{theorem}[Bourgain-Milman-Wolfson cube dichotomy~\cite{BMW}]\label{thm:BMW dich}
For any metric space $(X,d_X)$ one of the following two dichotomic
possibilities must hold true:
\begin{itemize}
\item either for all $n\in \N$ we have $c_X\left(\{0,1\}^n,\|\cdot\|_1\right)=1$,
\item or there exists $\alpha=\alpha(X)$ and $c=c(X)>0$ such that for all $n\in \N$ we have $c_X\left(\{0,1\}^n,\|\cdot\|_1\right)\ge cn^\alpha$.
\end{itemize}
\end{theorem}
 For more information on the theory of metric dichotomies see~\cite[Sec.~1.1]{MN-charlie}, the survey paper~\cite{Men09}, and the references therein. The most fundamental open question in this area concerns the sharpness of the metric cotype dichotomy (Theorem~\ref{thm:MN dich}), or even more generally, the possible rates of growth of $\{\D_n(\F)\}_{n=1}^\infty$. Bourgain's embedding theorem~\cite{Bourgain-embed} (combined with Dvoretzky's theorem~\cite{Dvo60}) says that $\D_n(X)\lesssim \log n$ for every infinite dimensional Banach space $X$. Linial, London and Rabinovich~\cite{LLR} proved that $\D_n(L_1)\asymp \log n$. This was extended by Matou\v{s}ek~\cite{Mat97} to all $L_p$ spaces $p\in [1,\infty)$, by showing that $\D_n(L_p)\asymp 1+\frac{1}{p}\log n$. The work of Ozawa~\cite{Ozawa} and Pisier~\cite{pisier-79,pisier-2008} shows that $\D_n(X)\asymp_X \log n$ for $X$ in a class Banach spaces satisfying certain geometric conditions; this class includes all Banach lattices with finite Rademacher cotype. Lafforgue's work~\cite{Laff08} shows that $\D_n(X)\asymp_X\log n$ whenever $X$ is a $K$-convex Banach space (see also~\cite{lafforgue-2009,MN10}). Additional results along these lines (when $X$ is not necessarily a Banach space) follow from~\cite{NS-2004}.

In light of Theorem~\ref{thm:MN dich} and the above quoted results, we recall the following natural open question from~\cite{MN-cotype-full}.
\begin{question}[Metric cotype dichotomy problem]
Does there exist a class of metrics spaces $\F$ for which $\lim_{n\to \infty}\D_n(\F)=\infty$ yet $\D_n(\F)=o(\log n)$? If so, for which $\alpha\in (0,1)$ there exists a class of metric spaces $\F=\F_\alpha$ such that $\D_n(\F)\asymp_\F (\log n)^\alpha$? What happens if we insist in these questions that $\F$ consists of a single Banach space $X$?
\end{question}

Theorem~\ref{thm:omega} shows that a judicious choice of metric transform $\omega$ cannot show that $\omega(\MET)$ solves the metric cotype dichotomy problem. Furthermore, Proposition~\ref{prop:transform-dichotomy-tight} shows that for every $\theta\in (0,1]$ we can have $\D_n(X)=n^\theta$ for some metric space $X$. Previously it was shown by Matou\v{s}ek~\cite{Mat-lowdim} that when $d$ is even integer, $\D_n\left(\ell_2^d\right)$ behaves
roughly like $n^{2/d}$ (up to polylogarithmic multiplicative factors). It is unknown what are the possible rates of growth of sequences such as $\{\D_n(\F)\}_{n=1}^\infty$; the only currently known restriction, from Theorem~\ref{thm:MN dich}, is that such sequences cannot be larger than $1$ yet behave like $(\log n)^{o(1)}$.


\section{A dichotomy theorem for line metrics}

The following simple result will be used in the proof of Theorem~\ref{thm:omega}. It implies that either every metric space $X$ contains arbitrarily large ``almost geodesics", or any embedding of the path $P_n$ into $X$ incurs very large distortion. Without quantitative estimates on the rate of growth of $c_X(P_n)$, such a result has been previously proved by Matou\v{s}ek in~\cite{Mat-BD} via a metric differentiation argument.

\begin{proposition} \label{prop:line-dichotomy}
For any class of metric spaces $\F$ one of the following two dichotomic
possibilities must hold true:
\begin{itemize}
\item for every $L\subseteq \R$ we have $c_\F(L)=1$,
\item or there exists $\alpha=\alpha(\F)>0$ and $n_0=n_0(\F)\in \N$ such that for all $n\ge n_0$ we have  $c_\F(P_n)\geq
n^\alpha$.
\end{itemize}
\end{proposition}

Proposition~\ref{prop:line-dichotomy} is a simple consequence of the following lemma, taken from~\cite[Prop.~5.1]{MN-charlie}. The proof of this lemma in~\cite{MN-charlie} is a ``baby version" of the sub-multiplicativity method that is commonly used in Banach space theory; see Pisier's characterization of trivial Rademacher type~\cite{Pis73} as an early example of many such arguments. A thorough discussion of the sub-multiplicativity method in the context of metric dichotomies in contained in~\cite{Men09}.

\begin{lemma}\label{prop:path-boosting}
Fix $\delta\in (0,1)$, $D\ge 2$ and $t,n\in \N$ satisfying $n\ge
D^{(4t\log t)/\delta}$. If $(X,d_X)$ is a metric space and $f:P_n\to X$ satisfies $\dist(f)\le D$,
then there exists  $\phi:P_t\to P_n$ which is a rescaled isometry, i.e., $\dist(\phi)=1$, such that
$\dist(f\circ \phi)\le 1+\delta$.
\end{lemma}

\begin{proof}[Proof of~Proposition~\ref{prop:line-dichotomy}]
If the first assertion of Proposition~\ref{prop:line-dichotomy} fails then there exists some
$\delta\in(0,1)$ and a finite $L_0\subseteq \mathbb R$ such that
$c_\F(L_0)> 1+\delta$.
Using dilation and rounding, there exists $t\in \mathbb N$
for which $c_{P_{t}}(L_0)\le 1+\delta/3$, and therefore
$c_\F(P_{t})> 1+\delta/2$. Define $D=n^{\delta/(8t \log t)}$ and assume that $n$ is large enough so that $D\ge 2$.  By
Lemma~\ref{prop:path-boosting} we obtain
\[ c_\F(P_n) \ge n^{{\delta}/{8 t \log t}} . \qedhere \]
\end{proof}


Proposition~\ref{prop:line-dichotomy} is sharp due to the following fact:
\begin{proposition} \label{prop:line-embed}
For every $\theta\in(0,1]$ there exists a metric space $X$ such that
for every finite subset of the real line $L\subseteq \mathbb R$ we have $c_X(L)\le (|L|-1)^{\theta}$, and
 $c_X(P_n)= n^\theta$ for all $n\in \N$.
\end{proposition}
\begin{proof}
Assume first that $\theta<1$. The space $X$ will be the snowflaked real line $\left(\R,|x-y|^{1-\theta}\right)$.

We first observe that any $(n+1)$-point subset $L\subseteq \mathbb R$ embeds in $X$ with distortion $n^\theta$. Indeed, write $L=\{y_0,\ldots, y_n\}$ where
 $y_0<y_1<\ldots<y_n$, and define $z_0=0$, and for $i\in \{1,\ldots,n\}$,
 $$
 z_i=\sum_{k=0}^{i-1} (y_{k+1}-y_k)^{1/(1-\theta)}.
 $$
 If $n\ge j> i\ge 0$ then,
 $$
 |z_j-z_i|^{1-\theta}=\left(\sum_{k=i}^{j-1}(y_{k+1}-y_k)^{1/(1-\theta)}\right)^{1-\theta}\in \left[y_j-y_i,\frac{y_j-y_i}{(j-i)^\theta}\right].
 $$
thus the embedding of $L$ into $X$ which maps $y_i$ to $z_i$ has distortion at most $n^\theta$.

The fact that $c_X(P_n)\ge n^\theta$ is simple. We briefly recall the standard computation. For a bijection $f:P_n\to X$ we have,
$$
  \frac{n}{\|f^{-1}\|_{\mathrm{Lip}}}
\leq |f(0) -f(n)|^{1-\theta}
\leq \left(\sum_{i=0}^{n-1} |f({i+1})-f(i)|\right)^{1-\theta}\le \left(n\|f\|_{\Lip}^{1/(1-\theta)}\right)^{1-\theta}=n^{1-\theta}\|f\|_{\Lip}.
$$
Therefore $\mathrm{dist}(f)\geq n^{\theta}$.

For $\theta=1$ let $\U$ be the class of all finite ultrametrics of diameter $1$. Let $X$ be the disjoint union of the elements of $\U$, equipped with the following metric: if $x,y\in X$ then let $U_1,U_2\in \U$ be finite ultrametics such that $x\in U_1$ and $y\in U_2$. Define $d_X(x,y)=1$ if $U_1\neq U_2$ and $d_X(x,y)=d_{U}(x,y)$ if $U_1=U_2=U$. Then $(X,d_X)$ is an ultrametric. It is a standard and easy fact (see for example~\cite[Lem.~2.4]{MN04} that any embedding of $P_n$ into an ultrametric incurs distortion at least $n$. Thus $c_X(P_n)\ge n$. Furthermore, it is well known (via a Cantor set-type construction; see for example~\cite[Lem.~3.6]{BLMN05}) that any $(n+1)$-point
metric space  embeds into some ultrametric $U\in \U$ with distortion at most $n$. Since $U$ is isometric to a subset of $X$, we have $c_X(P_n)=n$.
\end{proof}

\section{Proof of Theorem~\ref{thm:omega}}

We will deduce Theorem~\ref{thm:omega} from Proposition~\ref{prop:line-dichotomy} via the following lemma.

\begin{lemma}\label{lem:reduce to line}
Let $\omega:[0,\infty)\to [0,\infty)$ be a nondecreasing concave function  with $\omega(0)=0$. Then for every $n\in \N$ we have
\begin{equation}\label{eq:square}
\D_{n^2}\left(2^\R\hookrightarrow \omega(\MET)\right)\ge \D_n\left(\omega(\MET)\right).
\end{equation}
Here $2^\R$ denotes the class of metric spaces consisting of all nonempty subsets of the real line.
\end{lemma}

\begin{proof}
Fix an $n$-point metric space $(X,d)$ and consider the subset $L$ of the real line defined by $L=\{d(x,y)\}_{x,y\in X}$. Then $|L|\le n^2$. Therefore, if $D> \D_{n^2}\left(2^\R\hookrightarrow \omega(\MET)\right)$ we know that there exists a metric $\rho$ on $L$, and a scaling factor $\lambda>0$, such that for every $x,y,z,w\in X$,
\begin{equation*}\label{eq:lip omega condition}
\lambda|d(x,y)-d(w,z)|\le \omega(\rho(d(x,y),d(w,z)))\le D\lambda |d(x,y)-d(w,z)|.
\end{equation*}
For every $z\in X$ define a semi-metric $\rho_z$ on $X$ by $\rho_z(x,y)=\rho(d(x,z),d(y,z))$. Then $\delta=\max_{z\in X} \rho_z$ is also a semi-metric on $X$.  Using the monotonicity of $\omega$ (and the triangle inequality), for every $x,y\in X$ we have
\begin{multline*}
d(x,y)=\max_{z\in X}|d(x,z)-d(y,z)|\le \frac{1}{\lambda}\max_{z\in X} \omega(\rho(d(x,z),d(y,z)))\\
=\frac{1}{\lambda}\omega\left(\max_{z\in X} \rho(d(x,z),d(y,z))\right)=\frac{1}{\lambda} \omega(\delta(x,y)).
\end{multline*}
Similarly,
$$
d(x,y)\ge \frac{1}{D\lambda} \omega(\delta(x,y)).
$$
Thus the identity mapping between $(X,d)$ and $(X,\omega\circ \delta)\in \omega(\MET)$ has distortion at most $D$. Since $D$ was an arbitrary number bigger than $\D_{n^2}\left(2^\R\hookrightarrow \omega(\MET)\right)$, the proof is complete.
\end{proof}

\begin{proof}[Proof of Theorem~\ref{thm:omega}]
Failure of the first statement of Theorem~\ref{thm:omega} implies the existence of $n_0\in \N$ such that $\D_{n_0}(\omega(\MET))>1$. By Lemma~\ref{lem:reduce to line} it follows that $\D_{n_0^2}\left(2^\R\hookrightarrow \omega(\MET)\right)>1$. Thus the first dichotomic possibility of Proposition~\ref{prop:line-dichotomy} fails, forcing the conclusion that the second dichotomic possibility of Proposition~\ref{prop:line-dichotomy} holds, as required.
\end{proof}

\section{Snowflakes}

The quadratic dependence on $n$ in~\eqref{eq:square} can be removed when $\omega$ corresponds to a snowflake, leading to sharp bounds in this case. The appropriate variant of Lemma~\ref{lem:reduce to line} is as follows.

\begin{lemma}\label{lem:reduce to line snowflake}
For every $\theta\in (0,1)$ and every $n\in \N$ we have,
\begin{equation}\label{eq:improvment}
\D_{n}\left(2^\R\hookrightarrow (\MET)^\theta\right)= \D_n\left((\MET)^\theta\right).
\end{equation}
\end{lemma}

\begin{proof}
 We need to show that $\D_{n}\left(2^\R\hookrightarrow (\MET)^\theta\right)\ge \D_n\left((\MET)^\theta\right)$ (the reverse inequality is trivial). Fix an $n$-point metric space $(X,d)$ and  $D> \D_{n}\left(2^\R\hookrightarrow (\MET)^\theta\right)$. For every $z\in X$ consider the subset of the real line given by $L_z=\{d(x,z)
 \}_{x\in X}$. Since $|L_z|\le n$, there exists a metric $\rho_z$ on $L_z$ and a scaling factor $\lambda_z>0$ such that for every $x,y\in X$ we have
$$
 \lambda_z|d(x,z)-d(y,z)|\le \rho_z(d(x,z),d(y,z))^\theta\le D\lambda_z|d(x,z)-d(y,z)|.
 $$
Define a semi-metric $\delta$ on $X$ by $$\delta(x,y)=\max_{z\in X} \frac{\rho_z(d(x,z),d(y,z))}{\lambda_z^{1/\theta}}.$$
Then for all $x,y\in X$,
$$
d(x,y)=\max_{z\in X}|d(x,z)-d(y,z)|\in \left[\frac{\delta(x,y)^\theta}{D},\delta(x,y)^\theta\right].
$$
This means that $(X,d)$ is bi-Lipschitz equivalent with distortion at most $D$ to the $\theta$-snowflake of $(X,\delta)$.
\end{proof}

We can now state and prove a concrete version of Proposition~\ref{prop:transform-dichotomy-tight}.
\begin{corollary} \label{cor:poly-growth}
For every $\theta\in (0,1)$ and every $n\in \N$ we have
$\D_n\left((\MET)^\theta\right)=(n-1)^{1-\theta}$.
\end{corollary}

\begin{proof}
We have seen in the proof of Proposition~\ref{prop:line-embed} that any $n$-point subset of the real line embeds into the $\theta$-snowflake of $\R$ with distortion at most $(n-1)^{1-\theta}$, and that this bound is attained for $P_n$. Now apply Lemma~\ref{lem:reduce to line snowflake}.
\end{proof}

\begin{remark}\label{rem:diff}
An inspection of the proof of Lemma~\ref{lem:reduce to line snowflake} shows that if $\omega:[0,\infty)\to [0,\infty)$ is increasing, concave, and $\omega(0)=0$, then the improvement~\eqref{eq:improvment} over~\eqref{eq:square} holds provided the class of metric spaces $\omega(\MET)$ is closed under dilation, i.e., if $d\in \omega(\MET)$ and $\lambda>0$ then also $\lambda d\in \omega(\MET)$. Equivalently, $\omega^{-1}(\lambda\omega(d))$ is a metric for every $\lambda>0$ and every metric $d$. This is the same as requiring that the function $t\mapsto \omega^{-1}(\lambda\omega(t))$ is subadditive on $[0,\infty)$. We do not know the possible moduli $\omega$ satisfying this requirement, but we believe that it is quite stringent, and possibly it characterizes snowflakes. Here we note that if $\omega$ is smooth on $(0,\infty)$ and $\lim_{t\to\infty} \omega(t)=\infty$ (in addition to being increasing, concave, and $\omega(0)=0$) and $t\mapsto \omega^{-1}(\lambda\omega(t))$ is concave for all $\lambda>0$, then there exist $a>0$ and $b\in (0,1]$ such that $\omega(t)=at^b$ for all $t\ge 0$. To see this, denote $f_\lambda (t)=\omega^{-1}(\lambda\omega(t))$. Then $f_\lambda'(t)=\lambda\omega'(t)/\omega'\left(\omega^{-1}(\lambda\omega(t))\right)$, and therefore the concavity of $f_\lambda$ is equivalent to the requirement
\begin{equation}\label{eq:second der}
\forall t>0,\quad 0\ge \frac{f_\lambda''(t)}{\lambda}=\frac{ \omega''(t)\omega'\left(\omega^{-1}(\lambda\omega(t))\right)-\frac{\lambda[\omega'(t)]^2\omega''\left(\omega^{-1}(\lambda\omega(t))\right)}{\omega'\left(\omega^{-1}(\lambda\omega(t))\right)}}
{\left[\omega'\left(\omega^{-1}(\lambda\omega(t))\right)\right]^2}.
\end{equation}
Assuming the validity of~\eqref{eq:second der}, we know that for all $\lambda,t>0$,
\begin{equation}\label{eq:secon step der}
\omega''(t)\omega'\left(\omega^{-1}(\lambda\omega(t))\right)\le \frac{\lambda[\omega'(t)]^2\omega''\left(\omega^{-1}(\lambda\omega(t))\right)}{\omega'\left(\omega^{-1}(\lambda\omega(t))\right)}.
\end{equation}
Denoting $s=\omega^{-1}(\lambda\omega(t))$, we see that since $w'>0$, inequality~\eqref{eq:secon step der} implies that for all $s,t>0$ we have,
\begin{equation*}\label{eq:force const}
\frac{\omega(s)\omega''(s)}{[\omega'(s)]^2}\ge \frac{\omega(t)\omega''(t)}{[\omega'(t)]^2}.
\end{equation*}

Thus there exists $c\in \R$ such that for all $t>0$ we have $\omega(t)\omega''(t)=c[\omega'(t)]^2$. Equivalently, $(\log \omega')'=c(\log \omega)'$. Thus for some $K\in \R$ we have $\log \omega'-c\log \omega=K$, or $\omega'/\omega^c=e^K$. If $c\neq 1$ then since $\omega(0)=0$, it follows that $c<1$ and $\omega(t)=(1-c)^{1/(1-c)}e^{K/(1-c)}t^{1/(1-c)}$. Since $\omega$ is concave, necessarily $1/(1-c)\in (0,1]$, as required. The case $c=1$ is ruled out since the equation $\omega'/\omega=e^K$ cannot be satisfied by a concave function.\qed
\end{remark}

\begin{remark}
The shortest path metrics on certain (in some cases any) constant degree expander graphs have been the only tool used so far to rule out intermediate behavior in the metric cotype dichotomy problem, i.e., to exhibit that for certain classes of metric spaces $\F$, the sequence  $\{\D_n(\F)\}_{n=1}^\infty$ cannot have asymptotic growth to infinity of $(\log n)^\alpha$ (up to constant factors) for some $\alpha\in (0,1)$. A simple consequence of Corollary~\ref{cor:poly-growth} is that expanders are not always the ``worst case" spaces for metric dichotomy problems. Indeed, if $G=(V,E)$ is an $n$-vertex constant degree expander, then $c_{(\MET)^\theta}(G)$ grows like $(\log n)^{1-\theta}$ rather than the required $n^{1-\theta}$. More generally (since constant degree expanders have logarithmic diameter), if $G=(V,E)$ is an $n$-vertex unweighted graph and $d_G$ is its shortest path metric, then
$c_{(\MET)^\theta}(V,d_G)= \Delta^{1-\theta}$, where $\Delta$ is the diameter of $G$. This is true since $(V,d_G)$ contains $P_{\Delta+1}$ isometrically, and therefore by Proposition~\ref{prop:line-embed} we have $c_{(\MET)^\theta}(V,d_G)\ge \Delta^{1-\theta}$. In the reverse direction, the identity mapping
 $(V,d_G) \mapsto (V,d_G^\theta)$ has distortion at most $\Delta^{1-\theta}$.
\end{remark}

\bibliographystyle{abbrv}
\bibliography{dich}
\end{document}